\documentclass[12pt]{amsart}
\usepackage{url, amssymb, enumerate, upref}
\newtheorem{theorem}{Theorem}[section]
\newtheorem{lemma}[theorem]{Lemma}
\newtheorem{proposition}[theorem]{Proposition}

\theoremstyle{definition}

\theoremstyle{remark}

\newtheorem{conjecture}[theorem]{Conjecture}
\numberwithin{equation}{section}

\title[Segments of bounded power series]{Sharp bounds for some segments \\ of bounded power series}

\author{Leonid V. Kovalev}
\address{215 Carnegie, Department of Mathematics, Syracuse University, Syracuse, NY 13244, USA \\
ORCiD: 0000-0001-8002-7155}
\email{lvkovale@syr.edu}



\begin{document}
\baselineskip6.5mm

\begin{abstract} We obtain sharp upper bounds for three-term segments of a bounded power series. Along the way we show that the Taylor polynomials of a certain algebraic function do not vanish in the unit disk.  
\end{abstract}

\maketitle

\section{Introduction}\label{sec-introduction} 

Suppose that a holomorphic function $f(z)=\sum_{k=0}^\infty a_k z^k$ satisfies $|f(z)|\le 1$ for all $z$ in the open unit disk $\mathbb D$. A classical theorem of 
Landau~\cite{Landau} gives sharp upper bounds for the partial sums $S_n(z)= \sum_{k=0}^n a_k z^k$, see~\eqref{eq-Landau-thm} below. In this paper we consider the growth of more general segments of the power series, $S_{m, n}(z)=\sum_{k=m}^n a_k z^k$. Our main result, Theorem~\ref{thm-three-terms}, implies that 
\begin{equation}\label{eq-uniform-bound-triples}
\sup_{z\in \mathbb D} |S_{m, n}(z)| < \frac{1}{3} + \frac{2\sqrt{3}}{\pi} \approx 1.436  \quad \text{if } n-m = 2   
\end{equation}
and the constant in~\eqref{eq-uniform-bound-triples} is the best possible. More precisely, Theorem~\ref{thm-three-terms} gives the sharp upper bound for $|S_{n-2, n}|$ for each $n\ge 2$. For $|S_{n-1, n}|$ such bounds were determined by Sz\'{a}sz~\cite{Szasz1918}; they approach $4/\pi$ as $n\to\infty$.   

The Landau-Sz\'{a}sz method relates the problem of bounding $S_{m, n}$ to the zeros of Taylor polynomials of the function $\sqrt{1+z+\cdots +z^{n-m}}$. Under the condition $n-m=2$ in~\eqref{eq-uniform-bound-triples}, one requires the following result. 

\begin{theorem}\label{thm-no-zeros} The partial sums of the power series $\sqrt{1+z+z^2} = \sum_{k=0}^\infty b_k z^k$ have no zeros in $\overline{\mathbb D}$. 
\end{theorem}

The case $n-m=d>2$ hinges on the following conjecture. 

\begin{conjecture}\label{q-zeros} For $d>2$  
the partial sums of the power series $\sqrt{1+z+\cdots+z^d} = \sum_{k=0}^\infty b_k^{(d)} z^k$ have no zeros in $\overline{\mathbb D}$. 
\end{conjecture}

The zeros of Taylor polynomials have been extensively studied (e.g.,~\cite{Dvoretzky}, \cite{EdreiSaffVarga}, \cite{JansonNorfolk}, \cite{Szego1936}), usually in regard to their asymptotic distribution. They are often considered for hypergeometric power series such as $e^z$, $\sin z$, $(1+z)^p$, etc. The series in Theorem~\ref{thm-no-zeros} is not hypergeometric and its coefficients are neither of constant sign nor of decreasing magnitude (Lemma~\ref{lemma-signs-coeffs}).  

Rewriting the function $\sqrt{1+z+\cdots+z^d}$ as $\sqrt{(1-z^{d+1})/(1-z)}$ makes it clear that its real part is positive in $\mathbb D$. This suggests an approach to Conjecture~\ref{q-zeros} via the positivity of certain trigonometric polynomials. 

\begin{conjecture}\label{q-positive} 
For $d\ge 2$ the trigonometric polynomials $\mathcal T_n^{(d)}(t)=\sum_{k=0}^n b_k^{(d)} \cos kt$ are positive on $\mathbb R$. 
\end{conjecture}

By the maximum principle, Conjecture~\ref{q-positive} implies Conjecture~\ref{q-zeros}. The case $d=1$ of Conjecture~\ref{q-positive} follows from $\sum_{k=1}^\infty |b_k^{(1)}| = 1 = b_0^{(1)}$. Moving beyond this seems difficult, even for $d=2$. The literature on the trigonometric series with positive partial sums (e.g., \cite{Belov, BrownWilson2} and references therein) generally addresses the case of monotone coefficients, whereas for $d\ge 2$ the sequence $\{|b_k^{(d)}|\}$ is oscillating. 

Let us describe the prior results in more detail. The Hardy norm ($H^p$ norm) of a holomorphic function $f\colon \mathbb D\to\mathbb C$ is 
\[
\|f\|_p = \left(\sup_{0<r<1} \frac{1}{2\pi} \int_0^{2\pi} |f(re^{it})|^p\,dt \right)^{1/p}, \quad 1\le p<\infty,
\]
with $\|f\|_\infty = \sup_{\mathbb D} |f|$. Landau's inequality~\cite{Landau} can be stated as
\begin{equation}\label{eq-Landau-thm} 
\|S_n\|_\infty \le \sum_{k=0}^n \binom{-1/2}{k}^2    
\end{equation}
where $S_n(z)=\sum_{k=0}^n a_k z^k$ is a partial sum (section) of $f$ and  $\|f\|_\infty\le 1$. The right hand side of~\eqref{eq-Landau-thm} grows indefinitely as $n\to\infty$, which one can infer, via Parseval's theorem, from the fact that the $H^2$ norm of $(1-z)^{-1/2}$ is infinite. 

After some rotation of the unit disk, one can assume that the maximum of $|S_n|$ is attained at $1$. Thus, the problem of estimating $\|S_n\|_\infty$ amounts to estimating the coefficient sum $a_0+\dots+a_n$. Similarly, for $\|S_{m, n}\|_\infty$ one needs to estimate $a_m+\dots +a_n$. 

Generalizing Landau's theorem, Sz\'{a}sz proved the following.

\begin{theorem} \label{thm-szasz} \cite[Satz I]{Szasz1918}
Suppose that complex numbers $\mu_0,\dots, \mu_n$, $\lambda_0,\dots,\lambda_n$ are such that the polynomial $P_\lambda(z) = \lambda_0+\lambda_1 z+ \dots +\lambda_n z^n$ satisfies
\begin{equation}\label{eq-szasz-assumption}
P_\lambda(z)^2 
= \mu_0 + \mu_1 z + \dots + \mu_n z^n + o(z^{n}),\quad z\to 0.
\end{equation}
Then for any function $f(z)=\sum_{k=0}^\infty a_kz^k$ with $\|f\|_\infty \le 1$ we have  
\begin{equation}\label{eq-szasz-conclusion}
|\mu_n a_0 + \mu_{n-1}a_1+ \cdots + \mu_0 a_n|\le 
|\lambda_0|^2 + \cdots + |\lambda_n|^2.
\end{equation}
If, in addition, $P_\lambda$ does not vanish in $\mathbb D$, then equality is attained in~\eqref{eq-szasz-conclusion} by the Blaschke product 
$f(z) = P_\lambda^*(z)/P_\lambda(z)$ where $P_\lambda^*(z) = \bar \lambda_n + \cdots +\bar \lambda_0 z^n$ is the conjugate-reciprocal polynomial of $P_\lambda$. 
\end{theorem}

The studies of extremal problems on Hardy spaces go far beyond Theorem~\ref{thm-szasz}, but since the further developments will not be invoked here, we refer an interested reader to Chapter~8 of Duren's book~\cite{Duren}.

To estimate $\|S_n\|_\infty$, one sets $\mu_0=\cdots = \mu_n=1$ in Theorem~\ref{thm-szasz} which yields $P_\lambda(z)=(1-z)^{-1/2} + o(z^{n})$. The binomial expansion leads to~\eqref{eq-Landau-thm}, which is sharp since $P_\lambda\ne 0$ in $\overline{\mathbb D}$ by the Enestr\"om–Kakeya theorem~\cite{GardnerGovil}. 

To estimate $\|S_{m, n}\|_\infty$, one chooses $\mu_k=1$ for $0\le k\le n-m$ and $\mu_k=0$ for $n-m<k\le n$. Thus, $P_\lambda(z)=(1+z+\cdots +z^{n-m})^{1/2} + o(z^{n})$ in this case. Having observed this in~\cite[\S7]{Szasz1918}, Sz\'{a}sz proceeded to analyze the case $n-m=1$, when the binomial expansion yields the sharp bound~\cite[Satz II]{Szasz1918}  
\begin{equation}\label{eq-Szasz-binomial} 
\|S_{n-1,n}\|_\infty \le \sum_{k=0}^n \binom{1/2}{k}^2.     
\end{equation}
In contrast to~\eqref{eq-Landau-thm}, the right hand side of~\eqref{eq-Szasz-binomial} has a finite limit as $n\to\infty$. Indeed, by Parseval's theorem 
\[
\sum_{k=0}^\infty \binom{1/2}{k}^2 = 
\frac{1}{2\pi} \int_0^{2\pi} |1+e^{it}|\,dt
= \frac{1}{2\pi} \int_0^{2\pi} |2\cos (t/2)|\,dt = \frac{4}{\pi}.
\] 

In order to estimate $\|S_{n-2, n}\|_\infty$, or equivalently the sum $a_{n-2}+a_{n-1}+a_n$, one has to work with the power series 
\begin{equation}\label{eq-taylor-sqrt2}
\sqrt{1+z+z^2} = \sum_{k=0}^\infty b_k z^k. 
\end{equation}
The coefficients in~\eqref{eq-taylor-sqrt2} are less explicit than the binomial coefficients, although they can be quickly computed using the recurrence $2k b_{k}  = (3-2k)b_{k-1} + (6-2k)b_{k-2}$, see Lemma~\ref{lemma-recurrence-coeffs}. We can now state the main result of this paper.

\begin{theorem}\label{thm-three-terms} Let $f(z)=\sum_{k=0}^\infty a_k z^k$  and $S_{m, n}(z)=\sum_{k=m}^n a_k z^k$. If $\|f\|_\infty\le 1$, then for all $n\ge 2$
\begin{equation}\label{eq-Szasz-trinomial} 
\|S_{n-2,n}\|_\infty \le \sum_{k=0}^n b_k^2 < \frac{1}{3} + \frac{2\sqrt{3}}{\pi}     
\end{equation}
where $b_k$ is as in~\eqref{eq-taylor-sqrt2}. Equality is attained in the first part of~\eqref{eq-Szasz-trinomial} by $f(z)=z^n T_n(1/z)/T_n(z)$, where   
\begin{equation}\label{eq-taylor-poly}
T_n(z) = \sum_{k=0}^n b_k z^k.
\end{equation}
\end{theorem}

As noted above, the left hand side of~\eqref{eq-Szasz-trinomial} can be replaced by $|a_{n-2}+a_{n-1}+a_n|$. However, an upper bound for $|a_{n-2}|+|a_{n-1}|+|a_n|$ would be different. Indeed, by expanding the M\"obius transformation 
\[
\frac{z-1/2}{1-z/2} = -\frac{1}{2} + \frac{3}{4}z+\frac{3}{8}z^2+\cdots
\]
we find
$|a_0|+|a_1|+|a_2| = 13/8$ which exceeds the right hand side of~\eqref{eq-Szasz-trinomial}.

An upper bound for the sum of three non-consecutive coefficients would also be greater than~\eqref{eq-Szasz-trinomial}. For example, consider the Blaschke product 
\[
\frac{1+z+2z^3}{2+z^2+z^3} = \frac{1}{2} + \frac{z}{2} - \frac{z^2}{4} + \frac{z^3}{2} + \cdots 
\]
for which $a_0+a_1+a_3 = 3/2$. 
  
\section{Preliminaries}\label{sec-preliminary}

We begin with an observation that~\eqref{eq-taylor-sqrt2}, being a differentiably finite power series~\cite{Stanley}, admits a recurrence relation with polynomial coefficients. 

\begin{lemma}\label{lemma-recurrence-coeffs}
The coefficients in~\eqref{eq-taylor-sqrt2} satisfy the recurrence relation 
\begin{equation}\label{eq-recurrence-coeffs}
2k b_{k}  = (3-2k) b_{k-1} 
+ (6-2k) b_{k-2}, \quad k\ge 1    
\end{equation}
where $b_0=1$ and $b_{-1} = 0$. 
\end{lemma}

\begin{proof} Let $p(z)=1+z+z^2$. Differentiating $f(z)=\sqrt{p(z)}$, we obtain $2pf'-p'f = 0$. Hence 
\[
2(1+z+z^2) \sum_{k=1}^\infty k b_k z^{k-1}
 - (1+2z) \sum_{k=0}^\infty b_k z^k = 0. 
\]
Collecting the terms with $z^{k-1}$ yields
\[
2k b_k + 2(k-1)b_{k-1} + 2(k-2)b_{k-2}
- (b_{k-1} + 2b_{k-2}) = 0
\]
which is equivalent to~\eqref{eq-recurrence-coeffs}. 
\end{proof}

Since the factors $(3-2k)$ and $(6-2k)$ in~\eqref{eq-recurrence-coeffs} are negative for $k>3$, the coefficients $b_k$ have a pattern of oscillation, which is clarified by the following lemma. 

\begin{lemma}\label{lemma-signs-coeffs} The coefficients in~\eqref{eq-taylor-sqrt2} satisfy $b_k>0$ when $k$ is not divisible by $3$. Also, $b_k<0$ when $k$ is divisible by $3$, with the exception of $b_0=1$.
\end{lemma}

\begin{proof} Using~\eqref{eq-recurrence-coeffs} it is easy to find that the series~\eqref{eq-taylor-sqrt2} begins with 
\[1 + \frac{1}{2}z+\frac{3}{8}z^2 - \frac{3}{16}z^3 + \frac{3}{128}z^4 + \frac{15}{256}z^5 + \cdots \]
In particular, $0<b_4<b_5$. The rest of the lemma follows by induction from the following assertion. 

\textbf{Claim}. Suppose $k\ge 1$ and $0<b_k<b_{k+1}$. Then $b_{k+2}<0<b_{k+3}<b_{k+4}$. 

To prove the above claim, write $b_k=A$ and $b_{k+1}=A+B$ where $A, B>0$. The recurrence relation~\eqref{eq-recurrence-coeffs} yields  
\begin{equation}\label{eq-induction-1}
\begin{split}
b_{k+2} &= \frac{-(2k+1)(A+B) - (2k-2) A}{2k+4} \\ 
&= \frac{-(4k-1) A - (2k+1)B}{2k+4} < 0.
\end{split} 
\end{equation}
Using~\eqref{eq-recurrence-coeffs} and then~\eqref{eq-induction-1} we find
\begin{equation}\label{eq-induction-2}
\begin{split} 
b_{k+3} &= \frac{-(2k+3) b_{k+2} - 2k (A+B)}{2k+6}  \\
& = \frac{(2k+3)((4k-1) A + (2k+1)B) - 2k(2k+4)(A+B) }{(2k+4)(2k+6)}  \\
& = \frac{(4k^2 + 2k - 3)A + 3B}{(2k+4)(2k+6)} > 0.
\end{split} 
\end{equation}
One more step of recursion yields 
\begin{equation}\label{eq-induction-3}
\begin{split} 
b_{k+4} - b_{k+3} &= \frac{-(2k+5) b_{k+3} - (2k+2) b_{k+2}}{2k+8} - b_{k+3}  \\
& = -\frac{4k+13}{2k+8} b_{k+3} - \frac{2k+2}{2k+8} b_{k+2}. \\
\end{split} 
\end{equation}
In order to prove that~\eqref{eq-induction-3} is positive, it suffices to show that its coefficients of $A$ and $B$ are positive. From~\eqref{eq-induction-1}--\eqref{eq-induction-2}, the coefficient of $A$ in ~\eqref{eq-induction-3} is 
\[
\frac{-(4k+13)(4k^2 + 2k - 3) + (2k+2)(4k-1)(2k+6)}{(2k+4)(2k+6)(2k+8)} 
\]
where the numerator simplifies to $9(2k+3)>0$. Similarly, the coefficient of $B$ in ~\eqref{eq-induction-3} is 
\[
\frac{-3(4k+13) + (2k+2)(2k+1)(2k+6)}{(2k+4)(2k+6)(2k+8)} 
\]
where the numerator simplifies to $8k^3 + 36k^2 + 28k - 27 >0$. 
\end{proof}

Knowing the signs of the coefficients in~\eqref{eq-taylor-sqrt2}, we can estimate its partial sums at $z=1$. 

\begin{lemma}\label{lemma-taylor-at1} The Taylor polynomials~\eqref{eq-taylor-poly} satisfy $T_n(1)>\sqrt{3}/3$ for $n=0,1,\dots$
\end{lemma}

\begin{proof} A theorem of Zygmund~\cite[Theorem VI.3.6]{Zygmund} states that the Fourier series of a H\"older continuous function of bounded variation converges absolutely. Since the restriction of $f(z)=\sqrt{1+z+z^2}$ to the unit circle satisfies these assumptions, we have $\sum_{k=0}^\infty |b_k| < \infty$. Hence $T_n\to f$ uniformly on $\overline{\mathbb D}$. Let $\zeta = e^{2\pi i/3}$. By Lemma~\ref{lemma-signs-coeffs}, 
\[
\sum_{k=0}^{\lfloor n/3\rfloor} b_{3k} > 
\sum_{k=0}^\infty b_{3k} = \frac{f(1) + f(\zeta) + f(\zeta^2)}{3} = \frac{\sqrt{3}}{3}. 
\]
Furthermore, the difference 
$T_n(1) - \sum_{k=0}^{\lfloor n/3\rfloor} b_{3k}$ is positive, since any terms that remain in it after cancellation are positive by Lemma~\ref{lemma-signs-coeffs}. 
\end{proof}

\section{Proofs of main results} \label{sec-main-proofs}

Recall that the polynomials $T_n$ are defined by~\eqref{eq-taylor-poly}. In order to prove $T_n\ne 0$ in $\overline{\mathbb D}$, we will study the coefficients of the product $T_n(z)\sqrt{1-z}$. The hope is that since $T_n(z)$ approximates $\sqrt{1+z+z^2}$, the product will be close to $\sqrt{1-z^3}$ coefficient-wise. 

Let $c_0, c_1, \dots$ be the coefficients of the power series $\sqrt{1-z} = \sum_{k=0}^\infty c_k z^k$, that is 
\begin{equation}\label{eq-helper-explicit}
c_k = (-1)^k \binom{1/2}{k} = (-1)^k \frac{(1/2)^{\underline{k}}}{k!}.
\end{equation}
Here and below $x^{\underline{k}} = x(x-1)\cdots (x-k+1)$ denotes the falling factorial. 
It is clear from~\eqref{eq-helper-explicit} that, with the exception of $c_0=1$, all coefficients $c_k$ are negative, and 
\begin{equation}\label{eq-helper-recurrence}
\frac{c_{k}}{c_{k-1}} = \frac{k-3/2}{k},\quad k\ge 1.    
\end{equation}

\begin{lemma}\label{lemma-prod-poly} Let $Q$ be a polynomial of degree $n$ with $Q(1)\ne 0$. Then the coefficients of the power series  
$Q(z) \sqrt{1-z} = \sum_{k=0}^\infty \gamma_k z^k$ 
satisfy     
\begin{equation}\label{eq-prod-lemma-2}
\frac{\gamma_k}{c_k} = 
\frac{P(k)}{(k-3/2)^{\underline{n}}},\quad k=0,1,\dots
\end{equation}    
where $P$ is a polynomial of degree $n$ with leading coefficient $Q(1)$. 
\end{lemma}

\begin{proof} Writing $Q(z)=\sum_{s=0}^n q_s z^s$, we find that 
$\gamma_k = \sum_{s=0}^n q_s c_{k-s}$ with the convention that $c_{k-s}=0$ when $k<s$. The relation~\eqref{eq-helper-recurrence} implies 
\[
\frac{c_{k-s}}{c_k} = \frac{k^{\underline{s}}}{(k-3/2)^{\underline{s}}}
\]
which remains true even if $k<s$ when both sides vanish. Hence
\[
\frac{\gamma_k}{c_k} = \sum_{s=0}^n  \frac{q_s k^{\underline{s}}}{(k-3/2)^{\underline{s}}} = 
 \sum_{s=0}^n \frac{ q_s  k^{\underline{s}}(k-s-3/2)^{\underline{n-s}} }{(k-3/2)^{\underline{n}}}
\]
which matches~\eqref{eq-prod-lemma-2} with 
\begin{equation}\label{eq-prod-poly-formula}
P(x) =  \sum_{s=0}^n q_s  x^{\underline{s}}(x-s-3/2)^{\underline{n-s}}.   
\end{equation}
The coefficient of $x^n$ in~\eqref{eq-prod-poly-formula} is $\sum_{s=0}^n q_s = Q(1)$. 
\end{proof}

We are now ready to implement the idea described at the beginning of this section.

\begin{proposition}\label{prop-all-negative} Let $T_n$ be as in~\eqref{eq-taylor-poly}. The series
\begin{equation}\label{eq-prod-1}
T_n(z) \sqrt{1-z} = \sum_{k=0}^\infty \gamma_k z^k  
\end{equation}
satisfies $\gamma_k\le 0$ for $k>0$, and $\gamma_k < 0$ for $k>n$.     
\end{proposition}

\begin{proof} Since $T_n(z) = \sqrt{1+z+z^2} + o(z^n)$ as $z\to 0$, it follows that 
$T_n(z) \sqrt{1-z} = \sqrt{1-z^3} + o(z^n)$. Thus, for $k=1,\dots, n$ we have 
\begin{equation}\label{eq-small-coeffs}
\gamma_k = \begin{cases}
c_{k/3} & \text{if $3\mid k$} \\ 
0 & \text{otherwise}
\end{cases}
\end{equation}
where $c_{k/3}<0$ due to~\eqref{eq-helper-explicit}. 

Lemma~\ref{lemma-prod-poly} yields  
\begin{equation}\label{eq-quotient-coeffs}
\gamma_k = \frac{P(k) c_k}{(k-3/2)^{\underline{n}}},\quad k=0,1,\dots
\end{equation}
where $P$ is a polynomial of degree $n$ with leading coefficient $T_n(1)>0$ (recall Lemma~\ref{lemma-taylor-at1}). For $k\ge n+1$ we have $(k-3/2)^{\underline{n}}>0$ and $c_k<0$. Thus, it remains to show that $P(k)>0$ for $k\ge n+1$. Since the leading coefficient of $P$ is positive, it suffices to prove the following claim. 

\textbf{Claim.} The polynomial $P$ has $n$ real roots, all of which lie in the interval $[1, n+1)$. Here and below the roots are counted with multiplicity.

From~\eqref{eq-small-coeffs} and~\eqref{eq-quotient-coeffs} we see that $P$ vanishes at those integers $1,\dots, n$ that are not divisible by $3$. This proves the above claim for $n\le 2$. From now on $n>2$. 

If $1\le k\le n$ and $3\mid k$, the formulas~\eqref{eq-small-coeffs} and~\eqref{eq-quotient-coeffs} show that $P(k)$ has the same sign as $(k-3/2)^{\underline{n}}$. Since this falling factorial has $n-k+1$ negative factors, we have
\begin{equation}\label{eq-sign-k3}
(-1)^{n-k+1} P(k) > 0 \quad \text{if } 1\le k\le n \text{ and } 3\mid k.   
\end{equation}
It follows that $P(k)$ and $P(k+3)$  have opposite signs when $1\le k < k+3\le n$ and $3\mid k$. Therefore, the number of roots of $P$ on the interval $(k, k+3)$ is odd. Since $P(k+1)=0=P(k+2)$, the interval $(k, k+3)$ has at least three roots of $P$. 

So far we found that $P$ vanishes at $n - \lfloor n/3\rfloor$ integers and has $\lfloor n/3\rfloor - 1$ additional roots which lie in the intervals $(3, 6)$, $(6, 9)$, etc. Thus, the interval $[1, n]$ contains at least $n-1$ roots of $P$. To locate the remaining root, we consider three cases.

\textit{Case 1:} $n+1$ is divisible by $3$. From~\eqref{eq-prod-poly-formula} we find
\[
P(n+1/2) = \sum_{s=0}^n b_s (n+1/2)^{\underline{s}}(n-s-1)^{\underline{n-s}}  
=  b_n (n+1/2)^{\underline{n}} > 0.
\]
where the last step uses Lemma~\ref{lemma-signs-coeffs}. Since $P(n-2)<0$ by~\eqref{eq-sign-k3}, the interval $(n-2, n+1/2)$ contains an odd number of roots of $P$. These include $n-1$ and $n$, so there must be at least three roots. Thus, in Case~1 all roots of $P$ lie in $[1, n+1/2)$. 

\textit{Case 2:} $n+2$ is divisible by $3$. As in Case 1, we have $P(n+1/2)>0$. Since $P(n-1)>0$ by~\eqref{eq-sign-k3}, the interval $(n-1, n+1/2)$ contains an even number of roots of $P$. Since $P(n) = 0$, there is another root in this interval. We conclude that all roots of $P$ lie in $[1, n+1/2)$. 

\textit{Case 3:} $n$ is divisible by $3$. Then $P(n)<0$ by~\eqref{eq-sign-k3}. We claim that $P(n+1)>0$.
Indeed, from~\eqref{eq-quotient-coeffs} we have
\[
P(n+1) = \frac{\gamma_{n+1}}{c_{n+1}}(n-1/2)^{\underline{n}}
\]
where $c_{n+1}<0$ and $(n-1/2)^{\underline{n}} > 0$. It remains to show that $\gamma_{n+1}<0$. By adding $b_{n+1}z^{n+1}\sqrt{1-z}$ to both sides of~\eqref{eq-prod-1}, we obtain 
\begin{equation}\label{eq-prop-np1}
T_{n+1}(z)\sqrt{1-z} = b_{n+1}z^{n+1}\sqrt{1-z}
+ \sum_{k=0}^\infty \gamma_k z^k. 
\end{equation}
The left hand side of~\eqref{eq-prop-np1} is
$\sqrt{1-z^3} + o(z^{n+1})$ as $z\to 0$, which implies that the coefficient of $z^{n+1}$ in~\eqref{eq-prop-np1} is $0$. Hence
$\gamma_{n+1} = - b_{n+1} < 0$, using Lemma~\ref{lemma-signs-coeffs}. This proves $P(n+1)>0$. The change of sign of $P$ on the interval $(n, n+1)$ shows that all zeros of $P$ lie  in $[1, n+1)$. The proof of Proposition~\ref{prop-all-negative} is complete.    
\end{proof}

\begin{proof}[Proof of Theorem~\ref{thm-no-zeros}] The series~\eqref{eq-prod-1} converges absolutely for $|z|\le 1$ since the binomial series for $\sqrt{1-z}$ does. Letting $z=1$ in~\eqref{eq-prod-1} we find that 
$\sum_{k=0}^\infty \gamma_k = 0$. Proposition~\ref{prop-all-negative} implies  that  
$\sum_{k=1}^\infty |\gamma_k| = \gamma_0 = 1$. By the triangle inequality, the series~\eqref{eq-prod-1} does not vanish in $\overline{\mathbb D}\setminus \{1\}$. Since $T_n(1)>0$ by Lemma~\ref{lemma-taylor-at1}, the polynomial $T_n$ has no zeros in $\overline{\mathbb D}$.     
\end{proof}

\begin{proof}[Proof of Theorem~\ref{thm-three-terms}] As was described in \S\ref{sec-introduction}, we apply 
Theorem~\ref{thm-szasz} with $\mu_k=1$ for $k=0,1,2$ and $\mu_k=0$ for $k=3,\dots,n$. The polynomial $P_\lambda$ from~\eqref{eq-szasz-assumption} is precisely $T_n$ from~\eqref{eq-taylor-poly}, so $\lambda_k=b_k$ for $k=0,\dots, n$.  Thanks to Theorem~\ref{thm-no-zeros} we know that equality is attained in $\|S_{n-2,n}\|_\infty \le \sum_{k=0}^n b_k^2$ when $f$ is a Blaschke product with denominator $T_n$. 

By Parseval's theorem,   
$\sum_{k=0}^\infty b_k^2$ is the squared $H^2$ norm of $\sqrt{1+z+z^2}$, which is 
\[
\frac{1}{2\pi} \int_0^{2\pi} |1+e^{it}+e^{2it}|\,dt
= \frac{1}{2\pi} \int_0^{2\pi} |1+2\cos t|\,dt  
=\frac{1}{3} + \frac{2\sqrt{3}}{\pi}.
\]
The proof is complete. 
\end{proof}

\section*{Acknowledgment}

The author thanks Xuerui Yang for the discussions that led to Conjecture~\ref{q-zeros}.  

\bibliography{taylor-ref.bib} 
\bibliographystyle{plain} 
\end{document}